\documentclass[a4paper,11pt]{amsart}
\usepackage{amsfonts}
 \usepackage{enumerate}
 \usepackage{amsmath, amsthm}
 \usepackage{marginnote}
\usepackage{amscd}
\usepackage{amssymb}
\usepackage{pdfsync}
\usepackage[title]{appendix}
\usepackage{xy}
\xyoption{all}
\usepackage{latexsym,graphicx}
\usepackage[latin1]{inputenc}
\usepackage{xspace}
\usepackage{array}
\usepackage{newclude}
\usepackage{hyperref}
\hypersetup{pdfpagelabels,
                  plainpages=false,
                    colorlinks=true,       
                    linkcolor=red}
\renewcommand*{\HyperDestNameFilter}[1]{\jobname-#1} 
\numberwithin{equation}{section}
\usepackage[centering]{geometry}

\usepackage[nameinlink]{cleveref}

\newcommand{\noi}{\noindent}

 \theoremstyle{plain}
\newtheorem{theor}{Theorem}[section]

\newtheorem{prop}[theor]{Proposition}
\newtheorem{lem}[theor]{Lemma}

\newtheorem{cor}[theor]{Corollary}

\theoremstyle{remark}
\newtheorem{rem}[theor]{Remark}

\theoremstyle{plain}
\newtheorem{defi}[theor]{Definition}

\numberwithin{equation}{section}
\newcommand{\pos}{\textnormal{pos}}

\newcommand{\CC}{{\mathbb C}}
\newcommand{\RR}{{\mathbb R}}

\newcommand{\QQ}{{\mathbb Q}}

\newcommand{\ZZ}{{\mathbb Z}}
\newcommand{\VV}{{\mathbb V}}
\newcommand{\WW}{{\mathbb W}}

\newcommand{\G}{{\mathbf G}}
\newcommand{\HH}{{\mathbf H}}

\newcommand{\NN}{{\mathbb N}}

\newcommand{\Ga}{\Gamma}

\newcommand{\HL}{\textnormal{HL}}

\newcommand{\ti}[1]{\mbox{$\tilde{#1} $}}

\newcommand{\ol}{\overline}

\newcommand{\lo}{\longrightarrow}

\newcommand{\Hom}{{\rm Hom}}

\newcommand{\Sh}{{\rm Sh}}

\newcommand{\ad}{{\rm ad}}

\newcommand{\der}{{\rm der}}

\newcommand{\GL}{{\rm \bf GL}}

\newcommand{\MT}{{\rm \bf MT}}

\newcommand{\alg}{\textnormal{alg}}
\newcommand{\proj}{{\mathbb P}}

\newcommand{\Aut}{\textnormal{Aut}}
\newcommand{\an}{\textnormal{an}}

\newcommand{\bH}{{\mathbf H}}

\newcommand{\NL}{\textnormal{NL}}

\newcommand{\cA}{{\mathcal A}}

\newcommand{\cD}{{\mathcal D}}

\newcommand{\cV}{{\mathcal V}}
\newcommand{\cO}{{\mathcal O}}
\newcommand{\cU}{{\mathcal U}}
\newcommand{\cY}{{\mathcal Y}}

\newcommand{\Zar}{\textnormal{Zar}}

\newcommand{\red}{\textnormal{red}}
\newcommand{\Hod}{\textnormal{Hod}}

\newcommand{\Hdg}{\textnormal{Hdg}}
\newcommand{\ws}{\textnormal{ws}}
\newcommand{\norm}{\textnormal{nor}}
\newcommand{\nc}{\textnormal{nt}}

\begin{document}
\title{On the closure of the positive dimensional Hodge locus}
\author{B. Klingler and A. Otwinowska}
\thanks{B.K.'s research is supported by an Einstein Foundation's
  professorship}

\begin{abstract}
Given $\VV$ a polarizable variation of $\ZZ$-Hodge structures on a
smooth connected complex
quasi-projective variety $S$,  the Hodge locus for $\VV^\otimes$ is
the set of closed points $s$ of $S$ where the fiber $\VV_s$ has more Hodge tensors than the
very general one. A classical result of Cattani, Deligne and
Kaplan states that the Hodge locus for $\VV^\otimes$ is a countable union of closed irreducible algebraic
subvarieties of $S$, called the special subvarieties of $S$ for
$\VV$.

Under the assumption that the adjoint group of the
generic Mumford-Tate group of $\VV$ is simple we prove that the union of the special subvarieties for $\VV$
whose image under the period map is not a point is either a closed algebraic
subvariety of $S$ or is Zariski-dense in $S$.

This implies for instance the following typical intersection
statement: given a Hodge-generic closed irreducible algebraic subvariety $S$ of the 
moduli space $\cA_g$ of principally polarized Abelian varieties of 
dimension $g$, the union of the positive dimensional
irreducible components of the intersection of $S$ with the strict special subvarieties of
$\cA_g$ is either a closed algebraic subvariety of $S$ or is Zariski-dense in $S$.
\end{abstract}

\maketitle

\section{Introduction} \label{intro}

\subsection{Motivation: Hodge loci}
Let $(\VV_\ZZ, \cV, F^\bullet, \nabla)$ be a polarizable variation of $\ZZ$-Hodge
structure ($\ZZ$VHS) of arbitrary weight on a smooth
connected complex quasi-projective
variety $S$. Thus $\VV_\ZZ$ is a finite rank locally free $\ZZ_{S^\an}$-local system on
the complex manifold $S^\an$ associated to $S$; and $(\cV, F^\bullet,
\nabla)$ is the unique algebraic regular filtered flat connection on
$S$ whose analytification is $\VV \otimes_{\ZZ_{S^\an}} \cO_{S^{\an}}$ endowed
with its Hodge filtration $F^\bullet$ and the holomorphic flat connection
$\nabla^\an$ defined by $\VV$, see
\cite[(4.13)]{Schmid}). From now on we will abbreviate the $\ZZ$VHS
$(\VV_\ZZ, \cV, F^\bullet, \nabla)$ simply by $\VV$.

A typical example of such a $\ZZ$VHS, referred to as ``the
geometric case'', is the weight zero $\ZZ$VHS $(\VV_\ZZ:=R^{2k}f^\an_* \,
\ZZ(k)/ (\textnormal{torsion}), \cV:= 
R^{2k}f_* \Omega^\bullet_{X/S}, F^\bullet, \nabla)$ associated to a
smooth projective morphism of smooth irreducible  complex
quasi-projective varieties  $f:X \to S$. In this case the Hodge filtration $F^\bullet$ is
induced by the stupid filtration on the algebraic De Rham complex
$\Omega^\bullet_{X/S}$ and $\nabla$ is the Gau\ss\--Manin
connection.

The Hodge locus $\HL(S, \VV)$ is the set
of points $s \in S^\an$ for which the Hodge structure $\VV_{s}$
admits more Hodge classes than the very general fiber $\VV_{s'}$
(for us a Hodge class in a pure $\ZZ$-Hodge 
structure $H=(H_\ZZ, F^\bullet)$ is a class in $H_\ZZ$
whose image in $H_\CC$ lies in $F^0 H_\CC$, or equivalently a morphism of Hodge 
structures $\ZZ(0) \to H$). It is empty if $\VV$ contains no
non-trivial weight zero factor. More generally let $\VV^\otimes$ be
the countable direct sum of polarizable $\ZZ$VHSs $\bigoplus_{a, b \in
  \NN} {\VV}^{\otimes a} \otimes (\VV^\vee)^{\otimes b}$ (where $\VV^\vee$ denotes the $\ZZ$VHS dual of $\VV$).
The Hodge locus $\HL(S, \VV^\otimes)$ is the subset of points $s \in 
S^\an$ for which the Hodge structure $\VV_s$ admits more Hodge {\em tensors}
than the very general fiber $\VV_{s'}$. It contains 
$\HL(S, \VV)$, usually strictly.

In the geometric case  Weil \cite{Weil}
asked whether $\HL(S, \VV)$ is a countable union of closed
algebraic subvarieties of $S$ (he noticed that a positive
answer follows easily from the rational Hodge conjecture).
In \cite{CDK95} Cattani, Deligne and
Kaplan proved the following unconditional
celebrated result (we also refer to \cite{BKT} for an alternative
proof): 

\begin{theor}(Cattani-Deligne-Kaplan) \label{CDK}
Let $S$ be a smooth connected complex quasi-projective algebraic variety and
$\VV$ be a polarizable $\ZZ$VHS over $S$. 
Then $\HL(S, \VV)$ (thus also $\HL(S, \VV^\otimes)$) is a countable union of
closed irreducible algebraic subvarieties of $S$.
\end{theor}

The locus $\HL(S, \VV^\otimes)$ is easier to understand than $\HL(S,
\VV)$ as it has a group-theoretical interpretation. Recall that the
Mumford-Tate group $\MT(H) \subset \GL(H)$ of a $\QQ$-Hodge structure 
$H$ is the Tannakian group of the Tannakian category $\langle
H^\otimes\rangle$ of $\QQ$-Hodge structures
tensorially generated by $H$ and its dual 
$H^\vee$. Equivalently, the group $\MT(H)$ is the fixator in $\GL(H)$
of the Hodge tensors for $H$. Given a polarized $\ZZ$VHS $\VV$ on
$S$ as above and $Y \hookrightarrow S$ a closed irreducible algebraic subvariety, a
point $s$ of $Y^\an$ is said to be Hodge-generic
in $Y$ for $\VV$ if $\MT(\VV_{s, \QQ})$ has 
maximal dimension when $s$ ranges through $Y^\an$. Two Hodge-generic points in $Y^\an$ for $\VV$ have
the same Mumford-Tate group, called the generic Mumford-Tate group
$\MT(Y, \VV_{|Y})$ of $Y$ for $\VV$. The Hodge locus $\HL(S, \VV^\otimes)$ is
also the subset of points of $S$ which are not Hodge-generic in $S$
for $\VV$.

\begin{defi} \label{special}
A {\em special subvariety} of $S$ for $\VV$ is a closed irreducible
algebraic subvariety $Y \subset S$ maximal among the closed
irreducible algebraic subvarieties $Z$ of $S$ such that $\MT(Z,
\VV_{|Z}) = \MT(Y, \VV_{|Y})$.
\end{defi}

In particular $S$ is always special for $\VV$.
\Cref{CDK} for $\HL(S, \VV^\otimes)$ can be rephrased by saying that
the set of special subvarieties of $S$ for $\VV$ is countable
and that $\HL(S, \VV^\otimes)$ is the (countable) union of the strict special subvarieties of $S$ for
$\VV$.

\subsection{Main result}

In this paper we investigate the geometry of the
Zariski-closure of the Hodge locus $\HL(S, \VV^\otimes)$. Our methods
are variational, hence we only detect the special subvarieties
of $S$ for $\VV$ which are positive dimensional in the following sense:

\begin{defi}
A closed irreducible subvariety $Y$ of $S$ is said to be {\em
  positive dimensional for $\VV$} if the local system ${\VV}_{|Y}$ is not
isotrivial.
\end{defi}

\noindent
Equivalently, $Y$ is positive dimensional for $\VV$ if and
only if its algebraic monodromy group $\HH_Y$ for $\VV$ (see
\Cref{monodr}) is not equal to $\{1\}$; or equivalently if the period map $\Phi_S: S^\an
\to \Gamma \backslash \cD^+$ describing 
$\VV^\otimes$ (see \Cref{ws}) does not
contract $Y^\an$ to a point in the connected Hodge variety
$\Gamma \backslash \cD^+$. When $\VV$ satisfies the infinitesimal Torelli
condition (i.e. the period map $\Phi_S$ is an immersion), 
a closed irreducible subvariety $Y$ of $S$ is positive dimensional for
$\VV$ if and only if it is
positive dimensional. 

\begin{defi}
We define the positive dimensional Hodge locus $\HL(S, \VV^\otimes)_\pos \subset
\HL(S, \VV^\otimes)$ as the
union of the strict special subvarieties of $S$ for $\VV$
which are positive dimensional for $\VV$.
\end{defi}

Our main result describes the Zariski-closure of the positive Hodge
locus $\HL(S, \VV^\otimes)_\pos$:

\begin{theor} \label{main}
Let $\VV$ be a polarizable $\ZZ$VHS on a smooth connected complex quasi-projective variety
$S$. 
Suppose that the adjoint group of the generic Mumford-Tate group $\MT(S,
\VV)$ is simple (we will say that $\MT(S,
\VV)$ is non-product).
Then either $\HL(S, \VV^{\otimes})_\pos$ is a finite union of strict special
subvarieties of $S$; or it is Zariski-dense in $S$.
\end{theor}

In other words: either the set of strict special 
subvarieties of $S$ for $\VV$ which are positive dimensional
for $\VV$ has finitely many maximal elements (for
the inclusion); or the union of such special subvarieties is Zariski-dense in $S$.

\subsection{Examples}

\Cref{main} is new even in the much-studied case where the $\ZZ$VHS $\VV$
has weight $1$ or $2$. Let us warn the reader that these cases, which
are simpler to describe, are not representative: in higher weight we expect $\HL(S,
\VV^{\otimes})_\pos$ to be algebraic in general.

\subsubsection{Example 1: subvarieties of Shimura varieties} \label{ex1}

  Let $\Sh_K^0(\G, X)$ be a connected Shimura variety associated to a
  Shimura datum $(\G, X)$, with $\G$ non-product, and a level $K$
  chosen to be neat (we refer to \cite{Milne} for a nice survey on
  Shimura varieties). For $(\G, X) =
  (\mathbf{GSp}(2g), \HH_{g})$, $g\geq 1$, the Shimura variety $\Sh_K^0(\G,
  X)$ is the moduli space $\cA_g$ of principally
  polarized Abelian varieties of dimension $g$ (endowed with some
  additional level structure). Let $\VV$ be
  the $\ZZ$VHS on $\Sh_K^0(\G, X)$ associated to a faithful 
rational representation of $\G$ (see \cite[3.2]{EY03}). The Hodge locus
$\HL(\Sh_K^0(\G, X)):= \HL(\Sh_K^0(\G, X), \VV^\otimes)$ is well-known
to be independent of the choice of the faithful representation and
is completely described in terms of Shimura subdata of $(\G, X)$, see
\cite{Moo}. The special points of $\Sh_K^0(\G, X)$, i.e. the special
subvarieties of dimension zero, are the
CM-points, i.e. the points of $\Sh_K^0(\G, X)$ whose Mumford-Tate
group is commutative. In the case of $\cA_g$ the CM-points correspond
to abelian varieties with complex multiplication. Any connected
Shimura variety contains an analytically dense set of special 
points (see \cite[Lemma 3.3 and 3.5]{Milne}), in particular 
$\HL(\Sh_K^0(\G, X))$ is analytically dense in $\Sh_K^0(\G, X)$. The
same proof shows that $\HL(\Sh^0_K(\G, X))_\pos$ is analytically dense in
$\Sh^0_K(\G, X)$ as soon as it is not empty. For instance
$\HL(\cA_g)_\pos$ is analytically dense in $\cA_g$.

\begin{rem}
There exist Shimura varieties with empty
positive dimensional Hodge locus, for instance the Kottwitz unitary Shimura varieties (see
\cite{Clozel})  obtained by taking for $\G$ the group of invertible elements of a division algebra
of prime degree endowed with an involution of the second kind. Ball
quotients of Kottwitz type are the simplest examples.
\end{rem}

 If $S\subset \Sh^0_K(\G, X)$ is a closed irreducible subvariety the
 special subvarieties of $S$ for $\VV_{|S}$ are precisely the
 irreducible components of the intersection of $S$ with the special
 subvarieties of $\Sh^0_K(\G, X)$. \Cref{main} thus implies immediately:

 \begin{cor} \label{cor1}
Let $\Sh_K^0(\G, X)$ be a smooth connected Shimura variety associated to a
Shimura datum $(\G, X)$ with $\G$ non-product.
Let $S\subset \Sh^0_K(\G, X)$ be a closed irreducible
subvariety which is Hodge generic (i.e. $\MT(S, \VV_{|S})=
\G)$). 
Either the positive dimensional irreducible components of the intersection of
$S$ with the strict special subvarieties of $\Sh_K^0(\G, X)$ form a set with
finitely many maximal elements (for the inclusion), or their union is Zariski-dense in
$S$.
\end{cor}

\noi
In the case of $\Sh_K^0(\G, X)= \cA_g$ this reads:
\begin{cor} \label{cor1'}
  Let $S \subset \cA_g$ be a Hodge-generic closed irreducible
  subvariety. Either the set of positive dimensional closed irreducible subvarieties
of $S$ which are not Hodge generic has 
finitely many maximal elements (for the inclusion), or their union is Zariski-dense in
$S$.
\end{cor}

\Cref{cor1}, which describes the distribution of all positive
dimensional intersections of the Hodge generic $S$ with the special
subvarieties of $\Sh^0_K(\G, X)$, should be compared with the
classical Andr\'e-Oort conjecture, which
predicts under the same hypotheses that there are only finitely many special subvarieties of
$\Sh^0_K(\G, X)$ {\em contained} in $S$ and maximal for these properties. The Andr\'e-Oort conjecture
has been proven when $\Sh^0_K(\G, X)$ is of abelian type, for instance for $\Sh^0_K(\G, X)
= \cA_g$. We refer to \cite{KUY} for a survey on the Andr\'e-Oort conjecture. While the Andr\'e-Oort conjecture is an
``atypical intersection'' statement in the sense of \cite{Zannier},
\Cref{cor1} may be thought of as 
a ``typical intersection'' statement. In particular both statements
seem completely independent.

\smallskip
More generally \Cref{main} is
the ``typical intersection'' counterpart to the ``atypical
intersection'' conjecture for $\ZZ$VHS proposed in
\cite[Conj. 1.9]{klin} (which generalizes the Zilber-Pink
conjectures for Shimura varieties). It provides an answer to the
geometric part of the na\"ive \cite[Question 1.2]{klin} (we 
warn the reader that our $\HL(S, \VV^\otimes)$ is denoted $\HL(S,
\VV)$ in \cite{klin}).

\smallskip
Even in the setting of \Cref{cor1} or \Cref{cor1'}, we don't know of any
simple criterion for deciding whether $\HL(S, \VV_{|S}^\otimes)_\pos$
is a strict closed algebraic subvariety of $S$ or Zariski-dense in $S$. For $\Sh^0_K(\G, X) = \cA_g$,
Izadi \cite{Iz}, following ideas of \cite{ColPi}, proved that 
  $\HL(S, \VV_{|S}^\otimes)$ is analytically (hence Zariski-)
  dense in $S$  for any irreducible $S \subset \cA_g$ of codimension
  at most $g$. Her proof adapts immediately to show that $\HL(S, \VV_{|S}^\otimes)_\pos$ is analytically dense in $S$ if
$S$ has codimension at most $g-1$. Generalizing the results of \cite{Iz} to a general connected Shimura
variety $\Sh^0_K(\G, X)$, Chai (see \cite{Chai}) showed the
following. Let $\HH \subset \G$ be a Hodge subgroup. Let
$\HL(S,\VV^\otimes, \HH) \subset \HL(S, \VV^\otimes)$ denote the subset of points $s \in S$ whose
Mumford-Tate group $\MT_s(\VV)$ is $\G(\QQ)$-conjugated to $\HH$. Then
there exists an explicit constant $c(\G, X, \HH) \in \NN$, whose
value is $g$ in the example above, which
has the property that $\HL(S, \VV^\otimes, \HH)$, hence also $\HL(S,
\VV^\otimes)$ is analytically dense in $S$ as soon as $S$ has
codimension at most $c(\G, X, \HH)$ in $\Sh_K(\G, X)$. Once more it
follows from the analysis of the proof of \cite{Chai} that $\HL(S,
\VV^\otimes)_\pos$ is analytically dense in $S$ as soon as $S$ has codimension at
most $c(\G, X, \HH)-1$.

\subsubsection{Example 2: classical Noether-Lefschetz locus} \label{Example2}
Let $B \subset \proj H^0(\proj^3_\CC, \cO(d))$ be the open subvariety
parametrizing the
smooth surfaces of degree $d$ in $\proj^3_\CC$. From now on we suppose
$d>3$. The classical Noether theorem
states that any surface $Y \subset \proj^3_\CC$ corresponding to a
very general point $[Y] \in B$ has Picard group $\ZZ$:
every curve on $Y$ is a complete intersection of $Y$ with another
surface in $\proj^3_\CC$. The countable union $\NL(B)$ of closed algebraic
subvarieties of $B$ corresponding to surfaces with bigger Picard group
is called the Noether-Lefchetz locus of $B$.
Let $\VV \to B$ be the $\ZZ$VHS $R^2f_*\ZZ$, where $f: \cY \to
B$ denotes the universal family of surfaces of degree $d$. Clearly $\NL(B) \subset \HL(B,
\VV^\otimes)$. Green (see \cite[Prop.5.20]{Voisin}) proved
that $\NL(B)$ is analytically dense in $B$ (see also \cite{CHM} for a
weaker result). In
particular $\HL(B, \VV^\otimes)$ is dense in $B$. Once more
the analysis of Green's proof shows that in fact $\HL(B,
\VV^\otimes)_\pos$ is dense in $B$.
Now \Cref{main} implies the following:

\begin{cor} \label{cor2}
Let $S \subset B$ be a Hodge-generic closed irreducible subvariety.
Either $S \cap \HL(B, \VV^\otimes)_\pos$ contains only finitely many maximal positive
dimensional closed irreducible subvarieties of $S$, or the union of such subvarieties is Zariski-dense in $S$.
\end{cor}

\begin{rem}
  We don't know if \Cref{cor2} remains true if we replace $\HL(B,
  \VV^\otimes)_\pos$ with $\NL(B)$.
\end{rem}

\subsection{Ingredients and strategy for \Cref{main}}

Let us now describe the main ingredients and the strategy for the proof of \Cref{main}.
From now on we do not differentiate a
complex algebraic variety $X$ from its associated complex analytic
space $X^\an$, the meaning being clear from the context.

\subsubsection{On the Zariski-closure of the $F^i$-loci}

Given $\lambda \in \cV$ and $i \in \ZZ$ let $\VV^i (\lambda) \subset
\cV$ be the locus of $\cV$ where the flat transport of $\lambda$
belongs to $F^i \cV$; and let $S^i(\lambda):= p(\VV^i(\lambda))
\subset S$ be the locus of points of $S$ where some determination of the
flat transport of $\lambda$ at $s$ belongs to $F^i\cV$. Here $p: \cV
\to S$ denotes the natural projection. These
definitions are reviewed in details in \Cref{notations}.

\smallskip
When $i = 0$ and $\lambda \in \VV_\QQ$ is rational, $\VV^0(\lambda)$
is the locus where the flat transport of $\lambda$ is a rational Hodge
class. The precise version of \Cref{CDK} is that for $\lambda$
rational, $\VV^0(\lambda)$ is a closed algebraic
subvariety of $\cV$, finite over the finite union of special subvarieties
$S^0(\lambda)$. 

\smallskip
To study the Zariski-closure of $\HL(S, \VV^\otimes)$ the first idea of this paper consists in studying the geometry of
$S^i(\lambda)$ {\em for a general, not necessarily
rational,} $\lambda \in \VV_\CC$. In this generality the subsets $S^i(\lambda)$ are usually not even complex analytic
subvarieties of $S$, see \Cref{notations}. However we
manage to describe the Zariski-closure of any of their components (see \Cref{component} for the notion of
component of $S^i(\lambda)$):

\begin{theor} \label{closure}
For any $i \in \ZZ$ and any $\lambda \in \VV_\CC$, the Zariski-closure of
any of the (possibly infinitely many) components of
$S^i(\lambda)$ is a weakly special subvariety of $S$ for $\VV$.
\end{theor}

Here the weakly special subvarieties of $S$ for $\VV$ are a
generalisation, introduced in \cite{klin}, of the special subvarieties
of $S$ for $\VV$. See \Cref{defiWSperiod} for the original definition and \Cref{equi}
for a more geometric description. Any point of $S$ is weakly special
for $\VV$ hence \Cref{closure} says nothing for components of
$S^i(\lambda)$ which are points. On the other hand there are few
weakly special subvarieties of $S$ for $\VV$ of positive
dimension, hence \Cref{closure} provides a strong information on
positive dimensional
components of $S^i(\lambda)$.

\smallskip
\Cref{closure} is a result in functional transcendence. It follows
mainly from the Ax-Lindemann \Cref{Ax-Lindemann}  for $\ZZ$VHS
conjectured in \cite[Conj.7.6]{klin} as a
special case of \cite[Conj.7.5]{klin}, proven by Bakker-Tsimerman \cite[Theor. 1.1]{BT}.

\subsubsection{A global algebraicity result for the locus of classes
  of $F^i$-type}

The second ingredient in the proof of \Cref{main} is a global algebraicity statement for the union of the
$F^i$-loci of dimension bounded below. Precisely, for any integer $d\geq 0$, let $\VV^i_{\geq d} \subset F^i \cV$ be the
locus of classes $\lambda \in F^i \cV$ whose orbit under monodromy is
infinite and such that $\VV^i(\lambda)$ is ``of
dimension at least $d$ at $\lambda$'', see \Cref{component}. Let $S^i_{\geq
  d}(\VV) = p(\VV^i_{\geq d})$ be its projection in $S$.

\begin{theor} \label{A}
Let $\VV$ be a polarized $\ZZ$VHS on a smooth quasi-projective variety
$S$. 
For any $i \in \ZZ$ and any $d \in \NN^*$, the subset $\VV^i_{\geq d}
\subset F^i \cV$ is a closed
algebraic subvariety of $\cV$. Its projection $S^i(\VV)_{\geq d}$  is a closed
algebraic subvariety of $S$.
\end{theor}

In words: the property of a point $\lambda \in \cV$ of having a flat
leaf intersecting $F^i\cV$ in dimension at least $d>0$ is closed in the Zariski-topology.
\Cref{A} is in fact a special case of a more general result on algebraic flat
connections, see \Cref{Aplus}. It uses in a crucial way the properties
of parallel transport.


\subsubsection{Strategy for the proof of \Cref{main}}
Let us indicate how \Cref{main} follows from
\Cref{closure} and \Cref{A}.

\smallskip
First, using a finiteness
result of Deligne, we are reduced to showing that for $\VV$ a
polarizable $\ZZ$VHS with non-product generic Mumford-Tate group, the positive dimensional Hodge locus $\HL(S, \VV)_\pos$ is either a finite union of strict special subvarieties of $S$ for $\VV$
or is Zariski-dense in $S$. 

Let us assume for simplicity that the period map
$\Phi_S$ for $\VV$ is an immersion.  In that case the locus of exceptional
rational Hodge classes in $\cV$ is $ \VV_\QQ \cap \VV^0_{\geq 0}= \VV_\QQ \cap
F^0\cV$; the Hodge locus $\HL(S, \VV)$ is the projection $p(\VV_\QQ
\cap \VV^0_{\geq 0})$; and the positive dimensional Hodge locus $\HL(S, \VV)_\pos $ is the
projection $p (\VV_\QQ \cap \VV^0_{\geq 1}) \subset S^0_{\geq 1}(\VV)$.
The Zariski-closure $\ol{\HL(S, \VV)_\pos}^\Zar$ coincides with
$p(\ol{\VV^0_{\geq 1} \cap \VV_\QQ}^\Zar)$.

In \Cref{sat} we refine
\Cref{A} to show that there exists a non-empty Zariski open subset $\cU$ of $\ol{\VV^0_{\geq 1} \cap
\VV_\QQ}^\Zar \subset \cV$ such that for every point $\lambda \in \cU$ there exists a
component of $\VV^0(\lambda)$ of dimension at least $1$ contained in
$\ol{\VV^0_{\geq 1} \cap \VV_\QQ}^\Zar$. Projecting to $S$, there exists a non-empty
Zariski open subset $U$ of $\ol{\HL(S, \VV)_\pos}^\Zar$ such that for
every point $x \in U$ there exists a class $\lambda \in \cV$ and a
component of $S^0(\lambda)$ of dimension at least $1$ contained in
$\ol{\HL(S, \VV)_\pos}^\Zar$ and passing through $x$.

By \Cref{closure} the Zariski-closure of such a component of
$S^0(\lambda)$ is a positive dimensional weakly
special subvariety of $S$ for $\VV$. We thus obtain that there exists a non-empty
Zariski open subset $U$ of $\ol{\HL(S, \VV)_\pos}^\Zar$ such that for
every point $x \in U$ there exists a weakly special subvariety $Y_x$
of $S$ for $\VV$ contained in $\ol{\HL(S, \VV)_\pos}^\Zar$ and
passing through $\VV$.

Either one of these $Y_x$ equals $S$, hence $\ol{\HL(S, \VV)_\pos}^\Zar =S$. Otherwise the structure of weakly
special subvarieties and the assumption that $\MT(S, \VV)$ is
non-product imply that each $Y_x$ is contained in a strict special
subvariety $S_x$ of $S$ for $\VV$. As such an $S_x$ is contained in
$\HL(S, \VV)_\pos$ it follows that $\ol{\HL(S, \VV)_\pos}^\Zar=
\HL(S, \VV)_\pos$. But then $\HL(S, \VV)_\pos$ is a
finite union of special subvarieties.

The general case where $\Phi_S$ is not a submersion is dealt with similarly
using stratifications and the geometry of $S^0_d(\VV)$ for all
$d\geq 1$.

\subsubsection{A converse to \Cref{CDK}}
Recall that for $\lambda \in \VV_\QQ$ the precise version of
\Cref{CDK} states that $\VV^0(\lambda)$ is a closed
algebraic subvariety of $\cV$, finite over the
finite union of special subvarieties $S^0(\lambda)$. As a preliminary
to \Cref{closure}, \Cref{A} and \Cref{main}, we also provide for the
convenience of the reader the following kind of converse to \Cref{CDK}, which might be well-known to experts but which does
not seem to have appeared before.

\begin{prop} \label{converse}
Let $\lambda \in \cV$ and $i \in \ZZ$ be such that $\VV^i(\lambda)$ is a closed
algebraic subvariety of $\cV$. Then the projection $S^i(\lambda)$ of $\VV^i(\lambda)$ is a
finite union of special subvarieties of $S$. Moreover, $\VV^i(\lambda)$ is finite over $S^i(\lambda)$.
\end{prop}

\subsection{Organization of the paper}
The paper is organized as follows. \Cref{notations} mostly recalls the basic
definitions and properties of $\VV^i(\lambda)$, $S^i(\lambda)$
and their components. \Cref{ws} studies the geometric properties of the weakly special subvarieties of $S$ for $\VV$ needed
in the following sections. In particular we prove that they are
closed algebraic subvarieties, obtain a geometric
characterisation (\Cref{equi}), prove that they coincide
in fact with the bi-algebraic subvarieties of $S$ for the natural bi-algebraic
structure on $S$ defined by $\VV$ (see \Cref{bi-alg}, a result stated in \cite[Prop.7.4]{klin} without
proof), and state the Ax-Lindemann \Cref{Ax-Lindemann} for them. The following
sections provide the proofs of \Cref{converse}, \Cref{closure},
\Cref{A} and \Cref{main} successively.

\subsection{Acknowledgments}
We thank O. Benoist and J. Chen for their remarks on this work.

\section{Some notation} \label{notations}

\subsection{Notation for local systems} 

Let $S$ be a smooth connected complex quasi-projective
variety. Let $\VV_\ZZ$ be a finite rank locally free
$\ZZ$-local system on $S$ and $(\cV, \nabla)$ the regular
algebraic connection on $S$ \cite[Theor. 5.9]{Del} associated to $\VV_\ZZ$.

\smallskip
The local system $\VV_\ZZ$ can be uniquely written as $\tilde{S} \times_{\rho}
V_\ZZ$, where $\pi:\tilde{S} \to S$ denotes the complex analytic universal
cover of $S$ associated to the choice of a point $s_0$ in $S$, $V_\ZZ:=H^0(\tilde{S},
\pi^{-1} \VV_\ZZ)\simeq \VV_{s_{0}, \ZZ}$ is a free $\ZZ$-module of
finite rank and $\rho: \pi_1(S, s_{0}) \to
\GL(V_\ZZ)$ denotes the monodromy representation of the local system
$\VV_\ZZ$. This corresponds to a complex analytic trivialization of $\tilde{\cV} := \cV
\times_S \tilde{S}$ as a product $\tilde{S} \times V$, where $V: =V_\ZZ \otimes
_\ZZ \CC$. We still let $\pi: \tilde{S} \times V \to
\cV$ denote the natural projection.
Recall the following classical definition:

\begin{defi} \label{monodr}
Given a closed irreducible algebraic
subvariety $i: Y\hookrightarrow S$, let $n: Y^\norm \to Y$ be its
normalisation.  The algebraic monodromy group $\HH_Y$ of $Y$ for
$\VV_\ZZ$ is the (conjugacy class of the)
identity component of the Zariski-closure in $\GL(V_\QQ)$ of the monodromy of the
restriction to $Y^{\norm}$ of the local system $n^*\VV_\ZZ$.
\end{defi}

\begin{defi} \label{flat leave}
Given $\lambda = \pi (\tilde{s}, \lambda_0) \in \cV$ we define $\VV(\lambda):= \pi(\tilde{S} \times
\{\lambda_0\}) \subset \cV$ the flat leave of $\lambda$ for $\nabla$.
\end{defi}

The set $\VV(\lambda)$ is naturally a connected closed complex analytic subspace of the \'etal\'e
space of the complex local system $\VV_\CC:= \VV_\ZZ \otimes_\ZZ \CC$. We will always endow
$\VV(\lambda)$ with its reduced analytic structure. When $\lambda =\pi (\tilde{s}, \lambda_0) $ is not a complex multiple of an element of
$\VV_\ZZ$, the orbit of $\lambda_0$ in
$V$ under the monodromy group $\rho(\pi_1(S, s_{0})) \subset \GL(V)$
has usually accumulation points, in which case  $\VV(\lambda)$ is not
an analytic subvariety of $\cV$.

\subsection{Notation for $\ZZ$VHS}

Suppose now that $\VV:=(\VV_\ZZ, \cV, F^\bullet, \nabla)$ is a
$\ZZ$VHS on $S$. All $\ZZ$VHS are assumed to be polarizable. In
particular the algebraic monodromy group $\HH_S$ is semi-simple.

\begin{defi} \label{type}
  Let $\lambda \in \cV$ and $i \in \ZZ$. The locus of classes of $F^i$-type $\VV^i
  (\lambda)$ for $\lambda$ is the intersection of the flat leaf
  $\VV(\lambda)$ with $F^i\cV$:
  $$\VV^i(\lambda) := \VV(\lambda) \cap F^i \cV \subset F^i\cV\;\;.$$

The locus of $F^i$-type for $\lambda$ is the projection
$$S^i(\lambda):= p(\VV^i(\lambda))\subset S \;\;.$$
\end{defi}

Again, $\VV^i(\lambda)$ is naturally a complex analytic subspace
(possibly with infinitely many connected components) of the \'etal\'e
space of the complex local system $\VV_\CC:= \VV_\ZZ \otimes_\ZZ
\CC$; when $\lambda$ is not a complex multiple of an element of
$\VV_\CC$ the complex space $\VV^i(\lambda)$ is in general not an
analytic subspace of $\cV$; a fortiori its projection
$S^i(\lambda)\subset S$ is {\it a priori} not a complex analytic
subvariety of $S$.

\begin{rem}
  For $i=0$ and $\lambda \in \VV_\QQ$ the locus $\VV^0(\lambda)$ is also called the locus of
  Hodge classes for $\lambda$, usually denoted $\Hdg(\lambda)$; and
  $S^0(\lambda)$ is the Hodge locus of $\lambda$ considered by Weil, namely the locus $\HL(S, \lambda)$ of
points of $S$ where some determination of the flat transport of $\lambda$ becomes a Hodge
class.
\end{rem}

\begin{defi} \label{component}
Let $\lambda \in \cV$.
  \begin{itemize}
    
\item[(a)] A component of $\VV^i(\lambda)$ is an 
irreducible component of the complex analytic subvariety
$\VV^i(\lambda)$ of the \'etal\'e space of the complex local system 
$\VV$.

\item[(b)] A component of $S^i(\lambda)$ is the image under $p: \cV \to S$ of a
component of $\VV^i(\lambda)$.

\item[(c)]
For $\lambda \in V-\{0\}$, $i \in \ZZ$ and $d \in \NN$ let 
$\VV^i(\lambda)_{\geq d} \subset F^i \cV$, respectively
$S^i(\lambda)_{\geq d} \subset S$, be the union of 
components of $\VV^i(\lambda)$, resp. $S^i(\lambda)$, of dimension at
least $d$.
\end{itemize}
\end{defi}

\begin{rem} \label{proj}
Notice that for $\lambda \in \cV$ and $z \in \CC^*$, $\VV^i(z \lambda)= z \VV^i(\lambda)$ and
$S^i(z \lambda) = S^i (\lambda)$ for any 
$z \in \CC^*$. Hence, for $\lambda \in \cV$ not in the zero section, $S^i(\lambda)$ depends
only on $[\lambda] \in \proj V$.
\end{rem}

\smallskip
For $\lambda= \pi(\ti{s}, \lambda_0) \in \cV$ it follows from the theorem of the fixed part (see \cite[Cor. 7.23]{Schmid}) that $S^i(\lambda) \not
=S$ if and only if the $\HH_Y$-orbit of $\lambda_0$ in $V$ is
not reduced to a point, equivalently if and only if the orbit of $\lambda_0$ under $\rho(\pi_1(S,
s_{0})) \subset \GL(V)$ is infinite. We denote by $\VV_\QQ^\nc$ the direct factor of the local system
$\VV_\QQ$ corresponding to the sum of non-trivial irreducible
$\HH_Y$-factors of $V_\QQ$ (it is naturally a sub-$\QQ$VHS of
$\VV_\QQ$). By abuse of notation we write $\VV^{\nc}_\QQ - \{0\}$ for
$\VV^{\nc}_\QQ$ with the zero-section removed.

\begin{defi}
 We define the locus of non-trivial $F^i$-classes $$\VV^i_{\geq d}:= \bigcup_{\lambda \in \VV_\QQ^\nc -\{0\}} \VV^i(\lambda)_{\geq d} \subset
F^i\cV \quad \textnormal{and} \quad S^i(\VV)_{\geq d}:=
p(\VV^i_{\geq d}) \subset S\;\;.$$
\end{defi}

Thus the locus of non-trivial (rational) Hodge classes for $\VV$ is
$\Hdg(\VV) := \VV_\QQ \cap \VV^0_{\geq 0}$ and the Hodge locus $\HL(S,
\VV)$ is $p(\VV_\QQ \cap \VV^0_{\geq 0})$.

\section{Weakly special subvarieties and bi-algebraic geometry for
  $(S, \VV)$} \label{ws}

In this section we recall the definition of the weakly special subvarieties
of $S$ for $\VV$ given in \cite{klin}, study their geometry and prove
their bi-algebraic characterisation (stated in \cite{klin} 
without proof). We recall below the definitions of Hodge
theory we need and introduced in \cite{klin} (inspired by \cite{Pink89} and \cite{Pink05}),
and refer to \cite{klin} for more details.

\smallskip
Let $\G$ be the generic Mumford-Tate group of $S$ for $\VV$. Any Hodge generic point $s \in S$
defines a morphism of real algebraic groups $h_s: \CC^* \to \G_\RR$.
All such morphisms belong to the same connected component of a $\G(\RR)$-conjugacy
class $\cD$ in $\Hom (\CC^*, \G_\RR)$, which has a natural structure
of complex analytic space (see \cite[Prop.3.1]{klin}). The space
$\cD^+$ is a so-called Mumford-Tate domain, a refinement of the classical period domain for $\VV$ defined by Griffiths. The pair $(\G,
\cD^+)$ is a connected (pure) Hodge datum in the
sense of \cite[Section 3.1]{klin}, called the generic Hodge datum of
$\VV$. The $\ZZ$VHS $\VV$ is entirely described by its period map
$$\Phi_S: S \to \Hod^0(S,\VV) := \Gamma \backslash \cD^+\;\;,$$
where $\Gamma \subset \G(\ZZ)$ is a finite index subgroup and $\Hod^0(S,
\VV) := \Gamma \backslash \cD^+$ is the associated connected Hodge
variety (see \cite[Def. 3.18 and below]{klin}). We denote by
$\tilde{\Phi}_S: \ti{S} \to \cD^+$ the lift of $\Phi_S$.

\subsection{Weakly special subvarieties} \label{wsvarieties}
The weakly special subvarieties of $S$ for $\VV$ are defined in terms
of the weakly special subvarieties of the connected Hodge variety
$\Hod^0(S,\VV)$, which we first recall.

\subsubsection{Weakly special subvarieties of Hodge varieties}
Let $(\G, \cD^+)$ be a connected Hodge datum and $Y= \Gamma \backslash \cD^+$ an associated
connected Hodge variety. Hence $Y$ is an arithmetic quotient in the sense of
\cite[Section 1]{BKT} endowed with a natural complex analytic
structure (which is not algebraic in general). Recall that a Hodge morphism between
connected Hodge varieties is the complex analytic map deduced from a
morphism of the corresponding Hodge data (see \cite[Lemma 3.9]{klin}).
The special and weakly special subvarieties of $Y$ are irreducible
analytic subvarieties of $Y$ defined as follows (see \cite[Def.7.1]{klin}): 
\begin{defi} \label{defiWSperiod}
  Let $Y$ be a connected Hodge variety.

  \begin{itemize}
    \item[(1)] The image of any Hodge morphism $T \to Y$ between connected 
      Hodge varieties is called a special subvariety of $Y$. 
    \item[(2)] Consider any Hodge morphism $\varphi: T_1 \times T_2 \to Y$
      between connected Hodge varieties and any point $t_2 \in
      T_2$. Then the image $\varphi(T_1 \times \{t_2\})$ is called
      a weakly special subvariety of $Y$. It is said to be strict if
      it is distinct from $Y$.
\end{itemize}
\end{defi}

\begin{rem} \label{rem3}
  \cite[Def.7.1]{klin}, valid more generally for $Y$ a mixed Hodge
  variety and generalizing \cite[Def.4.1]{Pink05} to this context, gives the following
  apparently more general definition of a weakly special subvariety. Consider any Hodge morphisms $R
      \stackrel{\pi}{\leftarrow} T \stackrel{i}{\rightarrow} Y$
      between (possibly mixed) connected Hodge varieties and any point $r \in
      R$. Then any irreducible component of $i(\pi^{-1}(r))$ is called
      a weakly special subvariety of $Y$. When $Y$ is pure, i.e. $\G$
      is a reductive group, one easily checks that this definition
      reduces to \Cref{defiWSperiod}(2) above.
    \end{rem}

\begin{rem} \label{rem1}
Considering the connected Hodge variety $T_2=\{t_2\}$ associated to
the trivial algebraic group, any special subvariety
of $Y$ is a weakly special subvariety of $Y$.
\end{rem}

\begin{rem} \label{rem2}
As noticed in \cite[Rem. 4.8]{Pink05} in the case of Shimura varieties,
any irreducible component of an intersection of special  
(resp. weakly special) subvarieties of the Hodge variety $Y$ is a
special (resp. a weakly special) subvariety of $Y$. The proof is easy
and the details are left to the reader.
\end{rem}

\subsubsection{Weakly special subvarieties for $\VV$} \label{wsIntersection}

As in \cite[Prop. 3.20 and Def. 7.1]{klin} we define:

\begin{defi} \label{defiWS}
Let $p:\VV \to S$ be a $\ZZ$VHS over a quasi-projective complex
manifold $S$ with associated period map $\Phi_S:S \to \Hod^0(S, \VV)$.

Any irreducible complex analytic component of $\Phi_S^{-1}(Y)$, where
$Y$ is a special (resp. weakly special) subvariety of the connected
mixed Hodge variety $\Hod^0(S, \VV)$, is 
called a special (resp. weekly special) subvariety of $S$ for
$\VV$. It is said to be strict if it is distinct from $S$.
\end{defi}

Notice that an irreducible component of an intersection of special  
(resp. weakly special) subvarieties of $S$ for $\VV$ is not anymore necessarily a
special (resp. a weakly special) subvariety of $S$ for $\VV$: it might
happen that for $Y \subset \Hod^0(S,\VV)$ a special (resp. weakly special)
subvariety the preimage $\Phi_S^{-1}(Y)$ decomposes as a union $Z_1
\cup Z_2$ with $Z_i$, $i=1,2$ irreducible; in which case $Z_1$ and
$Z_2$ are special (resp. weakly special) subvarieties in $S$ but an irreducible
component of $Z_1 \cap Z_2$ is not. To take this minor inconvenience
into account we define more generally:
\begin{defi}
  Let $Y \subset \Hod^0(S, \VV)$ be a special (resp. weakly special)
  subvariety. An irreducible component of the intersection of some irreducible
  components of $\Phi_S^{-1}(Y)$ is called a special (resp. weakly
  special) intersection in $S$ for $\VV$.
\end{defi}

The following follows immediately from \Cref{rem2}:

\begin{lem} \label{inter}
  An irreducible component of an intersection of special (resp. weakly
  special) intersections for $\VV$ is a special (resp. weakly special)
  intersection for $\VV$.
\end{lem}

\subsubsection{Algebraicity of weakly special subvarieties of $S$}
The very definition of the Hodge locus $\HL(S, \VV^\otimes)$
implies that special subvarieties of $S$ for $\VV$ in the sense of
\Cref{defiWS} coincide with the ones defined in \Cref{special}. In
particular, in view of \Cref{CDK}, any special subvariety of $S$
(hence any special intersection in $S$) is
a closed irreducible algebraic subvariety of $S$. An alternative proof of \Cref{CDK} using o-minimal geometry was provided
in \cite[Theor. 1.6]{BKT}. The approach of \cite{BKT} gives immediately the following more
general algebraicity result, which is implicit in the discussion of \cite[Section
7]{klin}:

\begin{prop} \label{algebraicityWS}
 Any weakly special subvariety $Z$ for $\VV$ (hence also any
 weakly special intersection for $\VV$) is an algebraic subvariety of $S$. 
\end{prop}

\begin{proof}
  The proof is strictly analogous to the proof of \cite[Theor. 1.6]{BKT}.
  By \cite[Theor. 1.1(1)]{BKT} the Hodge variety 
  $\Hod^0(S, \VV)= \Gamma \backslash \cD^+$ is an arithmetic quotient endowed with a natural
  structure of $\RR_\alg$-definable manifold. By \cite[Theor. 1.3]{BKT} the period map $\Phi_S$ is $\RR_{\an,
    \exp}$-definable with respect to the natural $\RR_\alg$-structures
  on $S$ and $\Hod^0(S, \VV)$. Let $Y$ be the unique weakly special subvariety of $\Hod^0(S,
  \VV)$ such that $Z$ is an irreducible component of $\Phi_S^{-1}(Y)$. By \cite[Theor. 1.1(2)]{BKT}  $Y$ is an
  $\RR_\alg$-definable subvariety of $\Hod^0(S, \VV)$; hence its preimage $\Phi^{-1}_S(Y)$ is an $\RR_{\an,
    \exp}$-definable subvariety of $S$. By the definable Chow theorem of
  Peterzil and Starchenko \cite[Theor. 4.4 and Corollary 4.5]{PS}, the
  complex analytic $\RR_{\an,\exp}$-definable subvariety of the complex algebraic
  variety $S$ is necessarily an algebraic subvariety of $S$.
  Hence its irreducible complex analytic component $Z$ too.
\end{proof}

\subsubsection{Special and weakly special closure}

One deduces immediately from \Cref{inter} the following
\begin{cor} \label{wsclosure}
  Any irreducible algebraic subvariety $i: W \hookrightarrow S$ is contained in a
  smallest weakly special (resp. special) intersection $\langle W \rangle_{\ws}$ 
  (resp. $\langle W \rangle_{\textnormal{s}}$) of $S$ for $\VV$,
  called the weakly special (resp. special) closure of $W$ in $S$ for $\VV$.
  
\end{cor}

\begin{rem}
  Obviously $ W \subset \langle W \rangle_{\ws}\subset \langle W
  \rangle_{\textnormal{s}}\;\;.$
\end{rem}

The geometric description of $\langle W \rangle_{\textnormal{s}}$ is
easy. Let $(\G_W, \cD_W) \subset
  (\G, \cD)$ be the generic Hodge datum of the restriction of $\VV$ to
  the smooth locus of $W$. This induces a Hodge morphism of connected Hodge
  varieties $\varphi: \Gamma_W \backslash \cD_W^+ \to \Gamma \backslash
  \cD^+$, where $\Gamma_W:= \Gamma \cap \G_W(\QQ)$. The restriction
  of the period map $\Phi_S$ to the smooth locus of $W$ factorizes
  through the special subvariety $\varphi(\Gamma_W \backslash
  \cD_W^+)$ of $\Gamma \backslash \cD^+$ and we obtain:

  \begin{lem} \label{lem1}
  The special closure $\langle W \rangle_{\textnormal{s}}$ is the
  unique irreducible component of intersections of components of $\Phi_S^{-1} (\varphi(\Gamma_W
  \backslash \cD_W^+))$ containing $W$.
\end{lem}

The description of the weakly special closure $\langle W
\rangle_{\ws}$ is a bit more involved but similar to the one obtained by Moonen \cite[Section
3]{Moo} in the case of Shimura varieties.
Let $n:W^\norm \to W$ be the normalisation of $W$. Let
  $$\Phi_{W^{\norm}} : W^{\norm} \to \Gamma_W \backslash
  \cD_W^+$$ be the period map for $n^*\VV$. Hence we have a commutative diagram
  $$
  \xymatrix{
    W^{\norm} \ar[d]_n \ar[r]^>>>>>{\Phi_{W^{\norm}}} & \Gamma_W \backslash
\cD_W^+\ar[d]^\varphi \\
W   \ar[r]_>>>>>>{{\Phi_S}_{|W}} & \Gamma \backslash \cD^+ \;\;.}$$

Let $\bH_W$ be the algebraic monodromy group of $W$ for $\VV$. Thus $\bH_W$ is the identity component of
  the Zariski-closure of $(\Phi_{S} \circ n)_*(\pi_1(W^{\norm}))
  \subset \Gamma$ in $\GL(V)$. As $W^\norm$ is normal the open
  immersion $j: W^{\norm, 0} \hookrightarrow W^\norm$ of the smooth
  locus $W^{\norm, 0}$ of $W^\norm$ defines a surjection $j_*:
  \pi_1(W^{\norm, 0}) \to \pi_1(W)$. In particular $\bH_W$ is also the
  algebraic monodromy group of the restriction of $n^*\VV_\ZZ$ to
  $W^{\norm, 0}$. It thus follows from \cite[Theor.1]{An92}
  that $\bH_W$ is a normal subgroup of the derived group $\G_W^\der$. As $\G_W$ is reductive
  there exists a normal subgroup $\G'_W \subset \G_W$ such that $\G_W$ is
  an almost direct product of $\bH_W$ and $\G'_W$. In this way we
  obtain a decomposition of the adjoint Hodge datum $(\G_W^\ad, \cD_W^+)$ into a product
$$ (\G_{W}^\ad, \cD_W^+) = (\bH_{W}^\ad, \cD_{\bH_{W}}^+) \times ({\G'}_{W}^\ad,
\cD_{\G'_{W}}^+)\;\;,$$
inducing a decomposition of connected Hodge varieties $$ \Gamma_W
\backslash \cD_W^+= \Gamma_{\bH_{W}} \backslash \cD_{\bH_{W}}^+ 
\times \Gamma_{\G'_{W}} \backslash \cD_{\G'_{W}}^+\;\;.$$

\begin{lem} \label{proj}
  The projection of $\Phi_{W^{\norm}}(W^{\norm}) \subset
  \Gamma_W \backslash \cD_W^+$ on $\Gamma_{\G'_{W}} \backslash \cD_{\G'_{W}}^+$ is a single point $\{t'\}$.
\end{lem}

\begin{proof}
When $\Gamma \backslash \cD^+$ is a connected Shimura variety this
  is proven in \cite[Prop. 3.7]{Moo}. Moonen's argument does not extend
  to our more general situation: he uses that $\cD^+$ is a
  bounded domain in some $\CC^N$ in the Shimura case, which is not true for a general flag
  domain $\cD^+$. Instead we argue as follows. Choose any faithful linear representation $\rho: 
  {\G'}_{W}^\ad \to \GL(H)$ and a $\ZZ$-structure $H_\ZZ$ on the $\QQ$-vector
  space $H$ such that $\rho(\Gamma_{\G'_{W}}) \subset \GL(H_\ZZ)$. The $\ZZ$-local system on $W^{\norm}$ with
  monodromy representation
  $$ \lambda: \pi_1(W^{\norm}) \stackrel{(\Phi_{W^{\norm}})_*}{\lo}
  \Gamma_W \stackrel{{p_2}_*}{\lo} \Gamma_{\G'_{W}} \stackrel{\rho}{\lo} \GL(H_\ZZ)$$
  is a $\ZZ$VHS with period map
  $$ W^{\norm} \stackrel{\Phi_{W^{\norm, 0}}}{\lo} \Gamma_W
\backslash \cD_W^+ \stackrel{p_2}{\lo} \Gamma_{\G'^\ad}
    \backslash \cD_{\G'}^+\;\;.$$ By the very definition of the algebraic
    monodromy group $\bH_W$ the group $\lambda(\pi_1(W^{\norm}))
    \subset \GL(H_\ZZ)$ is finite. Applying the theorem of the fixed
    part (see \cite[Cor. 7.23]{Schmid}) to the corresponding \'etale
    cover of $W^{\norm}$ we deduce that the period map $p_2 \circ
    \Phi_{W^{\norm}}$ is constant.
  \end{proof}
  
\Cref{proj} implies that $W$ is contained in $\Phi_S^{-1}(\varphi((\Gamma_{\bH_{W}} \backslash \cD_{\bH_{W}}^+) \times
  \{t'\}))$. Conversely, as any irreducible component of an
  intersection of weakly special subvarieties of $\Gamma_W \backslash
  \cD_W^+$ is still weakly special, one easily checks that any weakly special
  subvariety $Y:= \psi(T_1 \times \{t_2\}) \subset \Gamma_W \backslash
  \cD_W^+$ containing $\Phi_{W^{\norm, 0}}(W^{\norm, 0})$ has to
  contain $(\Gamma_{\bH_{W}} \backslash \cD_{\bH_{W}}^+) \times \{t'\}$. Thus:
\begin{prop} \label{ws closure}
The weakly special closure $\langle W \rangle_{\ws}$ of $W$ is the
unique irreducible component of the intersection of components of
 $\Phi_S^{-1}(\varphi((\Gamma_{\bH_{W}} \backslash \cD_{\bH_{W}}^+) \times
  \{t'\}))$ containing $W$.
\end{prop}

It then follows immediately:
  
\begin{cor} \label{equi}
  The weakly special subvarieties of $S$ for $\VV$ (see \Cref{defiWS})
are precisely the closed irreducible
algebraic subvarieties $Y \subset S$ maximal among the closed
irreducible algebraic subvarieties $Z$ of $S$ whose algebraic
monodromy group $\HH_Z$ with respect to $\VV$ equals $\HH_Y$.
  
\end{cor}

\begin{rem}
  The reader will notice that the characterisation
  of the weakly special subvarieties given above is strictly analogous
  to the characterisation \Cref{special} of the special subvarieties,
  replacing the generic Mumford-Tate group by the algebraic monodromy
  group.
  \end{rem}

\subsection{Bi-algebraic geometry for $(S, \VV)$}

Let us start by recalling the general functional
transcendence context of ``bi-algebraic geometry'' (see
\cite{KUY}, \cite[Section 7]{klin}):

\begin{defi} \label{bialg}
A bi-algebraic structure on a connected complex algebraic variety $S$
is a pair 
$$ (D: \ti{S} \to X, \quad \rho: \pi_1(S) \to \Aut(X))\;\;$$ 
where $\pi:\ti{S} \to S$ denotes the universal cover of $S$, $X$ is a complex
algebraic variety, $\Aut(X)$ its group of algebraic
automorphisms, $\rho: \pi_1(S) \to \Aut(X)$ is a group morphism
(called the holonomy representation) and 
$D$ is a $\rho$-equivariant holomorphic map (called the developing map).
\end{defi}

The datum of a bi-algebraic structure on $S$ tries to emulate an
algebraic structure on the universal cover $\tilde{S}$ of $S$:
\begin{defi} \label{bialgebraic}
Let $S$ be a connected complex algebraic variety endowed with a
bi-algebraic structure $(D, \rho)$.
\begin{itemize}
\item[(i)]
An irreducible analytic subvariety $Z \subset \ti{S}$ is said to be a closed
irreducible algebraic subvariety of $\ti{S}$ if $Z$ is an analytic
irreducible component of $D^{-1}(\overline{D(Z)}^\Zar)$ (where
$\overline{D(Z)}^\Zar$ denotes the Zariski-closure of $D(Z)$ in $X$).
\item[(ii)]
A closed irreducible algebraic subvariety $Z\subset \ti{S}$, resp. $W
\subset S$, is said to be {\em bi-algebraic} if $\pi(Z)$ is a closed
algebraic subvariety of $S$, resp. any (equivalently one) analytic 
irreducible component of $\pi^{-1}(W)$ is a closed
irreducible algebraic subvariety of $\ti{S}$. 
\end{itemize}
\end{defi}

\begin{rem}
  As in \Cref{wsIntersection} an irreducible component of an
  intersection of closed algebraic subvarieties of $\ti{S}$ is not
  necessarily algebraic in the sense above, as the map $D$ is not assumed to be
  injective. Let us call such an irreducible component an {\em
    algebraic intersection} in $\ti{S}$. An algebraic intersection
  $Z\subset \ti{S}$, resp. a closed irreducible algebraic subvariety $W
\subset S$, is called a {\em bi-algebraic intersection} if $\pi(Z)$ is
a closed 
algebraic subvariety of $S$, resp. any (equivalently one) analytic 
irreducible component of $\pi^{-1}(W)$ is an algebraic intersection in
$\ti{S}$.
\end{rem}

Let $\VV$ be a polarized $\ZZ$VHS on $S$. It canonically defines a
bi-algebraic structure on $S$ as follows. Let $$\hat{\Phi}_S : \ti{S} \to
\hat{\cD}$$ be the composite $j \circ \tilde{\Phi}_S$ where $j: \cD
\hookrightarrow \hat{\cD}$ denotes the open embedding of the
Mumford-Tate domain $\cD$ in its compact dual $\hat{\cD}$, which is an
algebraic flag variety for $\G(\CC)$ (see \cite[Section 3.1]{klin}).

\begin{defi} \label{bi-alg def}
Let $p: \VV \to S$ be a polarized $\ZZ$VHS on a 
quasi-projective complex manifold $S$.
The bi-algebraic structure
on $S$ defined by $\VV$ is the pair $(\hat{\Phi}_S: \ti{S} \to
\hat{\cD}, \rho_S:= (\Phi_S)_*: \pi_1(S) \to \Ga \subset \G(\CC))$.
\end{defi}

The following proposition, stated in \cite[Prop. 7.4]{klin} without
proof, characterizes the weekly special subvarieties of $S$ for $\VV$ in
bi-algebraic terms. It was  proven by Ullmo-Yafaev \cite{UY1}
in the case where $S$ is a Shimura variety, and in some special cases
by Friedman and Laza \cite{FL15}.

\begin{prop} \label{bi-alg}
Let $(S, \VV)$ be a $\ZZ$VHS. The weakly special subvarieties (resp. the weakly special
intersections) of $S$ for $\VV$ are
the bi-algebraic subvarieties (resp. the bi-algebraic intersections) of $S$
for the bi-algebraic structure on $S$ defined by $\VV$.
\end{prop}

\begin{proof}
The proof is similar to the proof of \cite[Theor.4.1]{UY1}, we provide
it for completeness.

  Notice that the statement for the weakly special intersections follows immediately from the
  statement for the weakly special subvarieties. Hence we are reduced
  to prove that the weakly special subvarieties of $S$ coincide with the
  bi-algebraic subvarieties of $S$.

  That a weakly special subvariety of $S$ is bi-algebraic follows from
  the fact that a Hodge morphism of Hodge varieties $\varphi: T \to Y$
  is defined at the level of the universal cover by a closed analytic embedding $
  \cD_T^+ \hookrightarrow \cD_Y^+$ restriction of a closed algebraic immersion
 $\hat{\cD}_T \hookrightarrow \hat{\cD}_Y$.

  Conversely let $W$ be a bi-algebraic subvariety of $S$. With the
  notations of \Cref{ws closure} the period
  map ${\Phi_{S}}_{|W}: W \to \Hod^0(S, \VV)$ factorises trough the
  weakly special subvariety $\varphi((\Gamma_{\bH_{W}} \backslash
  \cD_{\bH_{W}}^+) \times \{t'\})$ of $\Hod^0(S, \VV)$. Let $Z$ be an
  irreducible component of the preimage of $W$ in $\tilde{S}$ and
  consider the lifting $\tilde{\Phi}_{|Z}: Z \to \cD_{\bH_{W}}^+$ of
  ${\Phi_{S}}_{|W}$ to $Z$. As $W$ is bi-algebraic the Zariski-closure
  of $\tilde{\Phi}_{|Z}(Z)$ in $\hat{\cD}_{\bH_{W}}$ has to be stable
  under the monodromy group $\bH_{W}(\CC)$, hence equal to
  $\hat{\cD}_{\bH_{W}}$.
  Thus $Z = (\tilde{\Phi}_{|Z})^{-1}(\cD_{\bH_{W}}^+)$ and $W$ is
  weakly special.
 \end{proof}

We will need the following result, proven for Shimura varieties in
\cite{KUY16}, conjectured in general in \cite[Conj.7.6]{klin} as a
special case of \cite[Conj.7.5]{klin}, and proven by Bakker-Tsimerman \cite[Theor. 1.1]{BT}:

\begin{theor}(Ax-Lindemann for $\ZZ$VMHS) \label{Ax-Lindemann}
Let $(S, \VV)$ be a $\ZZ$VMHS. Let $Y\subset \tilde{S}$ be a closed
algebraic subvariety for the bi-algebraic structure defined by
$\VV$. Then $\overline{\pi(Y)}^\Zar$ is a bi-algebraic 
subvariety of $S$, i.e. a weakly special subvariety of $S$ for $\VV$.
\end{theor}

\section{A converse to \Cref{CDK}: proof of \Cref{converse}} \label{proofconverse}

\begin{proof}[\unskip\nopunct]

  Let $f: S' \to S$ be a finite \'etale cover and let $\VV' := f^*
  \VV$. By abuse of notation let $f$ still denote the natural
  map $\cV' \to \cV$. The reader will immediately check the following
  (where, with the notations of \Cref{intro}, we naturally identify
  $V'$ with $V$):

\begin{lem}
\begin{itemize}
\item[(a)] $$ \forall \, \lambda \in V^*, \;\; \forall \; i \in \ZZ,
  \quad  {\VV'}^i(\lambda)= f^{-1} \VV^i(\lambda)\quad
  \textnormal{and}  \quad f({\VV'}^i(\lambda))= \VV^i(\lambda)\;\;.$$
\item[(b)]the $f$-image of a special subvariety of $S'$ for $\VV'$ is
  a special subvariety of $S$ for $\VV$; conversely the $f$-preimage
  of a special subvariety of $S$ for $\VV$ is a finite union of
  special subvarieties of $S'$ for $\VV'$.
\end{itemize}
\end{lem}

\noi
Hence proving \Cref{converse} for $\VV$ is equivalent to proving it for
$\VV'$. As any finitely generated linear group admits a torsion-free
finite index subgroup (Selberg's lemma) we can thus assume without loss of generality by
replacing $S$ by a finite \'etale cover if necessary that the
monodromy $\rho(\pi_1(S, s_{0})) \subset \GL(V_\ZZ)$ is torsion-free.

\smallskip
Let $\lambda \in V^*$ be such that $\VV^i(\lambda)$ is an
algebraic subvariety of $\cV$. Hence $\VV^i([\lambda]) \subset \proj
\cV$ is also algebraic. As the projection $p: \proj \cV \to S$ is a
proper morphism, it follows that the set
$S^i(\lambda):= p(\VV^i([\lambda])$ is an algebraic subvariety of
$S$. 

Let $n: S' \to S^i(\lambda)$ be the smooth locus of the normalisation
of one irreducible component of $S^i(\lambda)$. Hence $S'$ is
connected. Let $\pi': \tilde{S'} \to S'$ be its universal cover and
let $\rho': \pi_1(S', s'_0) \to \GL(V)$ be the 
monodromy of the local system $\VV':=i^{-1} \VV$ on
$S'$. Let $\widetilde{{\VV'}^i(\lambda)}:=
\pi'^{-1}({\VV'}^i(\lambda)) \subset 
\widetilde{\VV'(\lambda)}:=\pi'^{-1}(\VV'(\lambda)) \simeq
\tilde{S'} \times \{\lambda\}\subset \widetilde{\cV'}\simeq
\tilde{S'} \times V$. 

As $\widetilde{{\VV'}^i(\lambda)} \subset \tilde{S'} \times\{\lambda\}$
and $p: \widetilde{{\VV'}^i(\lambda)} \to \tilde{S'}$ is surjective, it
follows that $\widetilde{{\VV'}^i(\lambda)} = \tilde{S'}
\times\{\lambda\}$, hence ${\VV'}^i(\lambda) = \VV'(\lambda)$. 

In particular ${\VV'}^i(\lambda) \cap V= \VV'(\lambda) \cap
V= \rho(\pi_1(S', s'_{0}))\cdot \lambda \subset V$,
where $V$ is identified with $\cV'_{s'_{0}}$. As ${\VV'}^i(\lambda) \subset
\cV'$ is an algebraic subvariety, its fiber
${\VV'}^i(\lambda) \cap V$ is an algebraic subvariety of $V$. On the
other hand the set $\rho(\pi_1(S', s'_{0}))\cdot \lambda$ is
countable. Thus ${\VV'}^i(\lambda) \cap V$ is a finite set of
points, in particular $p: {\VV'}^i(\lambda) \to S'$ is finite \'etale. 

It follows that the smallest $\QQ$-sub-local system $\WW'_\QQ \subset
{\VV}'_\QQ$ whose complexification $\WW' \subset \VV'$
contains ${\VV'}^i(\lambda)$ has finite monodromy. As the monodromy
$\rho(\pi_1(S'))$ is a subgroup of $\rho(\pi_1(S)$ which is assumed to
be torsion-free, it follows that the local system $\WW'_\QQ$ is trivial. By the theorem of the
fixed part (see \cite[Cor. 7.23]{Schmid}) $\WW'_\QQ$ is a constant
sub-$\QQ$VHS of $\VV'$. It follows easily that $n(S')$ is the smooth
locus of an irreducible component of the Hodge locus in $S$ defined by
the fiber $W_\QQ\subset V_\QQ$ of $\WW'$. 

This finishes the proof that $S^i(\lambda)$ is a union of special
subvarieties of $S$ and that $p: \VV^i(\lambda) \to S^i(\lambda)$ is
finite.
\end{proof}

\section{Proof of \Cref{closure}} \label{proofclosure}

\Cref{closure} follows from \Cref{Ax-Lindemann} and the following

\begin{prop} \label{algeb}
Any component of $S^i(\lambda)$ is the image under $\pi$ of an
algebraic subvariety of $\tilde{S}$ (for the bi-algebraic structure on
$S$ defined by $\VV$).
\end{prop}

\begin{proof}[Proof of \Cref{algeb}]

The quadruple $(\VV_\ZZ, \cV, \nabla, F^\bullet)$ defining the
$\ZZ$VHS $\VV$ is the pullback under $\Phi_S$ of a similar quadruple
$(\VV_{\ZZ, \Gamma \backslash \cD^+} , \cV_{\Gamma \backslash \cD^+},
\nabla_{\Gamma \backslash \cD^+}, F^\bullet_{\Gamma \backslash
  \cD^+})$ on the connected Hodge variety $\Gamma \backslash \cD^+$, 
which however does not satisfy Griffiths transversality. This
quadruple itself comes, by restriction to $\cD^+$ and descent to
$\Gamma \backslash \cD^+$, from a $\G(\CC)$-equivariant quadruple
$(\VV_{\ZZ, \hat{\cD}} , \cV_{\hat{\cD}}, 
\nabla_{\hat{\cD}}, F^\bullet_{\hat{\cD}})$ on $\hat{\cD}$. As
$\hat{\cD}$ is simply connected the algebraic flat connection
$\nabla_{\hat{\cD}}$ induces a canonical algebraic trivialization
$\cV_{\hat{\cD}} \simeq \hat{\cD} \times V$. Hence we have 
a commutative diagram 
\begin{equation} \label{eq1}
\xymatrix{
\cV \ar[d]_{p} &\tilde{\cV}\simeq \tilde{S} \times V \ar[l]_>>>>>{\pi}
\ar[r]^>>>>>{\hat{\Phi}_S \times Id}
\ar[d]_{p} & \hat{\cD}_S \times V \simeq \cV_{\hat{\cD}} \ar[d]^{p_1} \\
S & \tilde{S} \ar[l]^{\pi} \ar[r]_{\hat{\Phi}_S } & \hat{\cD}_S \;\;.}
\end{equation}

Let $N$ be a component of $S^i(\lambda)$. Hence 
$N= \pi(Y)$, where $Y = p(W)$ for $W \subset \tilde{\cV} = \tilde{S}
\times V$ an analytic irreducible component of the complex analytic
subvariety 
$$\widetilde{\VV(\lambda)} :=(\hat{\Phi}_S \times Id)^{-1}( (\hat{\cD} \times
\{\lambda\}) \cap F^i \cV_{\hat{\cD}})$$
of $\tilde{\cV}$. 
Let us define $\hat{\cD}(\lambda) \subset \hat{\cD}$ as the projection $p_1((\hat{\cD} \times
\{\lambda\}) \cap F^i \cV_{\hat{\cD}})$. In particular, corresponding
to the diagram~(\ref{eq1}), we have a commutative
diagram 

\begin{equation} \label{eq2}
\xymatrix{\pi(W) \ar[d]_{p} & W \ar[l]_>>>>>{\pi}
\ar[r]^>>>>>{\hat{\Phi}_S \times Id}
\ar[d]_{p} &  (\hat{\cD} \times \{\lambda\}) \cap F^i \cV_{\hat{\cD}}\ar[d]^{p_1} \\
N & Y \ar[l]^{\pi} \ar[r]_{\hat{\Phi}_S } & \hat{\cD}(\lambda) }\;\;.
\end{equation}

Notice that both $(\hat{\cD} \times \{\lambda\})$ and $ F^i
\cV_{\hat{\cD}}$, hence also their intersection $(\hat{\cD} \times
\{\lambda\}) \cap F^i \cV_{\hat{\cD}}$, are algebraic subvarieties of
$\hat{\cD} \times V$. Thus their projection $\hat{\cD}(\lambda)$ is an
algebraic subvariety of $\hat{\cD}$. It follows that the irreducible
component $Y$ of $\hat{\Phi}_S^{-1}(\hat{\cD}(\lambda))$ is an
algebraic subvariety of $\tilde{S}$ for the bi-algebraic structure on
$S$ for $(S, \VV)$.

\end{proof}

\section{Algebraicity of the positive dimensional $F^i$-locus and
  relation with the $\QQ$-structure} \label{proofA}

\subsection{An algebraicity result for flat complex connections}

\begin{theor} \label{Aplus}
Let $d$ be a positive integer.
Let $p :(\cV, \nabla)\to S$ be an algebraic flat connection on a smooth
quasi-projective complex variety $S$ and $\VV \subset \cV$ the
associated complex local system. Let $F\subset \cV$ be an
algebraic subvariety. For $x \in F$ let $N_{F,x}$ denote the union of
irreducible components containing $x$ of the complex analytic subvariety $(\VV(x)
\cap F)^\red$ of the \'etal\'e space
of $\VV$.

The locus $A_{F, \geq d}$ of closed points $x\in
F$ such that $N_{F,x}$ has dimension at least $d$ at $x$, is an algebraic subvariety of $\cV$. 
\end{theor}

\begin{proof}
Let $T_h\cV \subset T\cV$ denote the horizontal algebraic subbundle of the
tangent bundle $T\cV$ defined by the flat connection $\nabla$. We write
$q:\proj (T\cV)\to \cV$ for the (proper) natural projection.
Let  $T_hF:=T_h\cV\times_{T\cV}TF$. We define inductively reduced algebraic varieties
$(A_{F,\geq d, n})_{n\in\NN}\subset \cV$ by
\begin{itemize}
  \item[-]
    $A_{F,\geq d, 0}:=F$,
   \item[-]
$A_{F,\geq d, n+1}:=\{ x \in A_{F,\geq d, n} \, |\, \dim ((T_h
A_{F,\geq d, n})_x) \geq d \} \;\;.$
\end{itemize}

Let $A_{F,\geq d,\infty}:=\bigcap_{n\in\NN}A_{F,\geq d, n}$.
As
the $A_{F,\geq d, n}$ are algebraic subvarieties of $\cV$, so is
$A_{F, \geq d, \infty}$.

From now on we write for simplicity $A_{\geq d}:=A_{F, \geq d}$, $A_{\geq d, n}:=
A_{F,\geq d, n}$, $A_{\geq d,\infty}:= A_{F,\geq d,\infty}$ and $N_x:=
N_{F,x}$. The
result then follows from \Cref{equality} below.
\end{proof}

\begin{lem} \label{equality}
The equality $A_{\geq d}=A_{\geq d, \infty}$ holds.
\end{lem}

\begin{proof}

The inclusion $A_{\geq d}\subset A_{\geq d, \infty}$ is equivalent to
the inclusions $A_{\geq d}\subset A_{\geq d, n}$ for all $n \in \NN$,
which we show by induction on $n$. By definition $A_{\geq d} \subset
F= A_{\geq d, 0}$. Assume that
$A_{\geq d} \subset A_{\geq d, n}$ for some $n\in \NN$. By definition
of $A_{\geq d}$, for any
$x\in A_{\geq d}$ the variety $A_{\geq d}$ contains an 
irreducible component $N$ of $N_{F, x}$ through $x$ of dimension at
least $d$. Hence
$$ d \leq \dim(N) \leq \dim (T_xN) \leq  \dim ((T_h A_{\geq d,
  n})_x)\;\;,$$
hence $x \in A_{\geq d, n+1}$. This shows $A_{\geq d} \subset A_{\geq
  d, n+1}$ and finishes the proof by induction that $A_{\geq d}\subset
A_{\geq d, \infty}$.

\smallskip
Conversely let us prove that $A_{\geq d, \infty} \subset A_{\geq d}$. 
Let $h: \tilde{\cV} \to V$ denote the composition
$$h: \tilde{\cV} \simeq \tilde{S} \times V \stackrel{p_2}{\to} V$$
  (where the first isomorphism is provided by the flat trivialisation).
For $x \in \cV$ and $\tilde{x} \in 
\pi^{-1}(x) \subset \tilde{\cV} \simeq \tilde{S}
\times V$ let $N_{\tilde{x}}$ be the union of the
irreducible components passing through $\ti{x}$ of the complex analytic
subvariety $h^{-1}(h(\tilde{x})) \cap \pi^{-1}(F)$ of $\tilde{\cV}$. Thus the local
biholomorphism $\pi: \tilde{\cV} \to \cV$ identifies $N_{\tilde{x}}$ locally at $\ti{x}$ with
$N_{x}$ locally at $x$. 

By noetherianity there exists an $n \in \NN$ such that $A_{\geq d, n} =
  A_{\geq d, n+1}=A_{\geq d, \infty}$. Hence for any $x \in A_{\geq d,
    \infty} $ we have $\dim ((T_h A_{\geq d, \infty})_x) \geq d$. Let
  us consider the restriction
$$h_{|\tilde{A}_{\geq d, \infty}}:\tilde{A}_{\geq d, \infty}\to V$$ of $h$ to
  $\tilde{A}_{|\geq d, \infty}:=A_{\geq
    d,\infty}\times_S\tilde{S}$. Let $U_{\geq d, \infty} \subset
  A_{\geq d,\infty}$ be the Zariski-dense open subset of smooth points
  $x$ of $A_{\geq d,\infty}$ such that the complex analytic map
  $h_{|\tilde{A}_{\geq d, \infty}}$ is smooth and locally
    submersive onto its image at any $\tilde{x} \in 
\{x\}\times_S\tilde{S}$. Hence, for $x \in U_{\geq d, \infty}$, 
\begin{equation} \label{dim}
  \dim_{\tilde{x}} (\tilde{A}_{\geq d, \infty} \times_V h(\tilde{x}) )
=
\dim ((T_h A_{\geq d, \infty})_x) \geq d\;\;.
\end{equation}
Since $U_{\geq d, \infty}$ is Zariski-dense in $A_{\geq d,\infty}$ the
inequality $\dim_{\tilde{x}} (\tilde{A}_{\geq d, \infty} \times_V
h(\tilde{x}))  \geq d$ holds for any $ \tilde{x}$ in the preimage $\tilde{A}_{\geq d,
  \infty}$ of $A_{\geq d,\infty}$ in $\tilde{\cV}$. As $A_{\geq
  d,\infty} \subset F$, any analytic irreducible component of $\tilde{A}_{\geq d, \infty} \times_V
h(\tilde{x})$ containing $x$ is contained in $N_{\tilde{x}}$. Thus
\Cref{dim} implies that for any $x \in A_{\geq
  d,\infty}$ we have $\dim_x (N_x) \geq d$, i.e. $x \in A_{\geq d}$.
\end{proof} 

\subsection{Applications to $\QQ$-local systems}
The following saturation result will be crucial in the proof of
\Cref{main}:

\begin{prop} \label{saturation} In the situation of \Cref{Aplus}
  suppose moreover that the local system $\VV= \VV_\QQ \otimes_\QQ
  \CC$ is defined over $\QQ$.
  Let $A_{F, \geq d, \QQ}:=A_{F, \geq d} \cap \VV_\QQ$ be the
  locus of rational classes $x$ whose flat transport meets $F$ at $x$ in
  dimension $\geq d$ and let ${A'}_{F, \geq d}:=\overline{A_{F, \geq
      d, \QQ}}^\Zar$ be its Zariski-closure in $\cV$. 
There exist a Zariski-open dense subset $U$ of ${A'}_{F, \geq d}$ and,
for
each $x \in U$, a component $N_{F, x}^0$ of $N_{F,x}$ of dimension at
least $d$ such that
$U \subset\bigcup_{x\in U}N^0_{F, x}\subset {A'}_{F, \geq d}$.
\end{prop}

\begin{proof}

As in the proof of \Cref{Aplus} we remove from now on the reference to
  $F$ in our notations.
  
  Notice first that $A_{\geq d, \QQ} \subset A_{\geq d} $ hence
${A'}_{\geq d}  \subset A_{\geq d} $ as $A_{\geq d}$ is
algebraic by \Cref{Aplus}.

Let $W$ be an irreducible component $W$ of ${A'}_{\geq d}$. Replacing $d$ by the largest
$d' \geq d$ such that $W= W \cap A_{\geq d'}$ we can
without loss of generality assume that $W \cap A_{\geq d+1}$ is a strict
closed algebraic subvariety of $W$. Let $U\subset W$ be the
Zariski-open dense subset of all $x\in W- (W \cap A_{\geq d+1})$ such that the variety ${A'}_{\geq d}$ is
smooth at $x$ and the morphism  
$h_{|\tilde{W}}:\tilde{W}\to V$ is locally submersive onto its image at 
any $\tilde{x} \in \{x\}\times_S\tilde{S}$. The fibers of the morphism
$h_{|\tilde{U}}:\tilde{U}\to V$ are smooth, let us call $D$ their
common dimension.

As $U$ is Zariski-open dense in $W$ and $W \cap A_{\geq d,
  \QQ}$ is Zariski-dense in $W$, there exists a point $x_0 \in U \cap
A_{\geq d, \QQ}$. The fiber of $h_{|\tilde{U}}$ at $\tilde{x_0}$
coincides near $\tilde{x}_0$ with a component of  $N_{\tilde{x}_{0}}$
of dimension d, hence $d=D$. The
fiber of $h_{|\tilde{U}}$ at $\tilde{x_0}$ contains a component of $N_{\tilde{x_{0}}}$ of dimension $d$. Hence $D=
d$.

For any $\tilde{x}\in \tilde{U}$ we have on the other hand
$$D=\dim_{\tilde{x}}\left(\tilde{W}\times_{H_s}\tilde{x}\right)
      =\dim_{\tilde{x}}(N_{\tilde{x}}\cap\tilde{W})
      \leq\dim_{\tilde{x}}(N_{\tilde{x}})=d = D\;\;.$$
Hence for any $x \in U$, $\dim_{\tilde{x}}(N_{\tilde{x}}\cap\tilde{W})
      =\dim_{\tilde{x}}(N_{\tilde{x}})$. Hence there exists a component $N_x^0$ of $N_{x}$ of dimension
      $d$ such that $N^0_{x}\cap U$ is open and dense in $N^0_{x}$. 
Hence $U$ is dense in $\bigcup_{x\in U}N^0_{x}$ (for the usual topology)
and $\bigcup_{x\in U}N^0_{x}\subset W$.

As this holds for any irreducible component $W$ of $A'_{\geq d}$, the
result follows.

\end{proof}

\subsection{Application to $\ZZ$VHS: proof of \Cref{A} and
corollary for Hodge loci}

Suppose now that $\VV$ is a $\ZZ$VHS and $F= F^i \cV$. Then $A_{F,
  \geq d} = \VV^i_{\geq d}$ and \Cref{Aplus} in this case is 
\Cref{A}.

\medskip
Moreover $A_{F,\geq d, \QQ} = \VV^i_{\geq d} \cap \VV_\QQ$
and \Cref{saturation} reads:

\begin{prop} \label{sat}
Let $S$ be a smooth complex quasi-projective algebraic variety and
$\VV$ be a polarized $\ZZ$VHS over $S$. Let $i \in \ZZ$ and $d \in \NN$.
There exist a Zariski-open dense subset $U$ of $\overline{\VV^i_{\geq d}
  \cap \VV_\QQ}^\Zar$ and, for each $\lambda \in U$, a component
$\VV^{i,0}(\lambda)$ of $\VV^i(\lambda)$ of dimension at least $d$
such that   
$U \subset \bigcup_{\lambda\in U} \VV^{i,0}(\lambda) \subset \VV^i_{\geq d}\;\;.$
\end{prop}

Consider now the Zariski-closure $\overline{p(\VV_\QQ \cap \VV^0_{\geq
    d})}^\Zar$. It coincides with the projection $p(\overline{\VV^i_{\geq d}
  \cap \VV_\QQ}^\Zar)$.
For $\lambda \in F^i\cV$ the projection $p(\VV^{i,0}(\lambda))$ is a
component of dimension at least $d$ of
$S^i(p(\lambda))$. By \Cref{closure} the Zariski-closure of any such
components is a weakly special subvariety of $S$ of dimension at least
$d$. We thus obtain

\begin{cor} \label{satenbas}
Let $S$ be a smooth complex quasi-projective algebraic variety and
$\VV$ be a polarized $\ZZ$VHS over $S$. Let $d \in
\NN$.
Then $\overline{p(\VV_\QQ \cap \VV^0_{\geq d})}^\Zar$ contains a
Zariski-open dense set $U$ with the following property: for each point
$x \in U$ there exists a weakly special subvariety $Y_x \subset
\overline{p(\VV_\QQ \cap \VV^0_{\geq d})}^\Zar$ of dimension at least
$d$ passing through $x$.
\end{cor}

\section{Proof of \Cref{main}} \label{final}

\begin{proof}[\unskip\nopunct]

Following Deligne (see \cite[Theor. 4.10]
{Voisin2}), there exists a bound on the tensors
one has to consider for defining $\HL(S, \VV^\otimes)$. Thus
$\HL(S,\VV^\otimes)= \bigcup_{i=1}^n \HL(S, \VV_i)$ for finitely many
irreducible weight zero $\ZZ$VHS $\VV_i \subset
\VV^\otimes$. It follows 
that $\HL(S,\VV^\otimes)_\pos = \bigcup_{i=1}^n \HL(S,
\VV_i)_\pos$. Hence, replacing $\VV$ by $\VV \oplus \bigoplus_{i=1}^n \VV_i$ if
necessary (this does not change the generic Mumford-Tate group, the
period map, or the special subvarieties), we are reduced without loss of generality  to
showing that for $\VV$ a polarizable $\ZZ$VHS the positive Hodge locus $\HL(S, \VV)_\pos$ is either a finite
union of special subvarieties of $S$ for $\VV$ or Zariski-dense in $S$.

To make the proof of \Cref{main} more transparent we deal first with special
cases.

\smallskip
\noi
{\em Case 1: the period map
$\Phi_S$ is an immersion.} In that case 
$$\HL(S, \VV)_\pos = p( (\VV_{\QQ} \cap (\VV)^0_{\geq 1})
\;\;.$$

Applying \Cref{satenbas} for $d=1$ to $\VV$
it follows that $\overline{\HL(S,\VV)_\pos}^\Zar$ contains a
Zariski-open dense subset $U$ with the following property: for each point $x
\in U$ there exists a
positive dimensional weakly special subvariety $W_x$ for $(S, \VV)$ passing
through $x$ and contained in $\overline{\HL(S,\VV)_\pos}^\Zar$. 

Either there exists $x \in U$ such that $W_x=S$, in which case
$\overline{\HL(S,\VV)_\pos}^\Zar=S$.
Or for all $x \in U$ the weakly special subvariety $W_x$ of $S$ is
strict. In this case the assumption that $\MT(S,\VV)$ is non-product
and the description of weakly special subvarieties given in
\Cref{wsvarieties} implies that each
$W_x$ is contained in a unique strict positive dimensional special 
subvariety $S_x$ of 
$S$. As $S_x$ belongs by definition to $\HL(S,\VV)_\pos$, it follows in this case that
$ \overline{\HL(S,\VV)_\pos}^\Zar =
\HL(S,\VV)_\pos$
is a finite union of strict special subvarieties of $S$, hence the
result. 

\smallskip
\noi
{\em Case 2: the period map $\Phi_S$ has constant relative dimension
  $d$.}
The proof is the same as in the first case,
replacing $(\VV_i)^0_{\geq 1 }$ and
``positive dimensional'' by $(\VV_i)^0_{\geq d }$ and ``at least
$(d+1)$-dimensional''.

\smallskip
\noi{\em General case:}
As the period map $\Phi_S$ is definable in the o-minimal structure
$\RR_{\an, \exp}$ 
(see \cite{BKT}), it follows from the trivialization theorem \cite[Theor. (1.2)
p.142]{VDD} that the locus $S_d \subset S$ where the fibers of
$\Phi_S$ are of complex dimension at least $d$ is an $\RR_{\an,
  \exp}$-definable subset of $S$. As $S_d$ is also a closed complex analytic
subset of $S$, if follows from the o-minimal Chow theorem
\cite[Theor.4.4 and Cor. 4.5]{PS} of
Peterzil-Starchenko that $S_d$ is a closed algebraic subvariety of
$S$. Finally we obtain an algebraic filtration $S=S_{d_{0}} \supsetneq
S_{d_{1}} \supsetneq \cdots \supsetneq S_{d_{k}} \supsetneq
S_{d_{k+1}}= \emptyset$.

Suppose that $\HL(S, \VV)_\pos$ is not algebraic. Let 
$i \in \{0, \cdots, k\}$ be the smallest integer such that $(S_{d_{i}}-S_{d_{i+1}})
\cap \HL(S, \VV)_\pos$ is not a closed algebraic subvariety of
$S_{d_{i}}-S_{d_{i+1}}$. 
As $\HL(S-S_{d_{i+1}}, \VV_{|S-S_{d_{i+1}}}^\otimes)_\pos = \HL(S,
\VV)_\pos \cap (S-S_{d_{i+1}})$, to prove that $\HL(S,
\VV)_\pos$ is Zariski-dense in $S$ we can and will assume without loss of generality
that $i=k$ (replacing $S$ by $S-S_{d_{i+1}}$ if necessary).

Without loss of generality we can assume that
$\HL(S,\VV)_\pos$ is contained in $S_{d_{i}}$: this is clear
if $i=0$, as $S= S_{d_{0}}$ in this case; if $i>0$ there are only finitely many maximal positive special subvarieties $Z_1,
\dots, Z_m$ of $S$ for $\VV$ intersecting
$S_{d_{i-1}}-S_{d_{i}}$ and we can without loss of generality replace $S$ by $S-(Z_1 \cup
\cdots \cup Z_m$).

Thus $\HL(S,\VV)_\pos$ coincide with $p((\VV)^0_{\geq d_{i}+1} \cap \VV_\QQ)$.
Applying \Cref{satenbas} with $d=d_i+1$, it follows that 
the union $Z$ of irreducible components of $\ol{\HL(S,\VV)_\pos}^\Zar$ contains a Zariski-open dense set
$U$ such that for every point $x \in U$ there exists a weakly special
subvariety $W_x$ of $S$ for $\VV$ of dimension at least $d_i +1$ passing
through $x$ and contained in $Z$. 

If $i >0$ the weakly special subvariety $W_x \subset Z \subset S_{d_{i}}$ is 
strict, and we conclude as above: each $W_x$ is contained in a unique
strict positive special subvariety $S_x$ for $
\VV$, thus $Z=\HL(S,\VV)_\pos$, which contradicts the assumption that $\HL(S, \VV)_\pos$ is not an algebraic subvariety of
$S$.

Thus $i=0$. Hence we are in Case 2 above and we conclude that $\HL(S, \VV)_\pos$ is
Zariski-dense in $S_{d_{0}}=S$.
This finishes the proof of \Cref{main}.
\end{proof}

\medskip
\noindent Bruno Klingler : Humboldt Universit\"at zu Berlin

\noindent email : \texttt{bruno.klingler@hu-berlin.de}

\medskip
\noindent Ania Otwinowska: Humboldt Universit\"at zu Berlin

\noindent email : \texttt{anna.otwinowska@hu-berlin.de}

\end{document}